\documentclass{amsart}

\usepackage{xcolor}
\usepackage{tikz}

\usepackage{marginnote}

\usepackage{amsmath}
\usepackage{amsthm}
\usepackage{amsrefs}
\RequirePackage{amssymb}
\usepackage{graphicx} % Required for inserting images

\newtheorem{theorem}{Theorem}[section]
\newtheorem{lemma}[theorem]{Lemma}

\theoremstyle{definition}
\newtheorem{definition}[theorem]{Definition}

\theoremstyle{remark}
\newtheorem{remark}[theorem]{Remark}

\numberwithin{equation}{section}

\newcommand{\Sc}{\mathcal{S}}

\newcommand{\Rc}{\mathcal{R}}
\newcommand{\vu}{\vec{u}}

\newcommand{\sym}{\text{Sym}}
\newcommand{\de}{\partial}

\DeclareMathOperator{\im}{\mathfrak{Im}}

\title{Planarity criteria for metric graphs}

\author{Alice Brolin}

\author{Pavel Kurasov}

\address{Dept. of Mathematics, Stockholm Univ., 106 91 Stockholm Sweden}
\email{alice.brolin@math.su.se}

\address{Dept. of Mathematics, Stockholm Univ., 106 91 Stockholm Sweden}

\email{kurasov@math.su.se} 

\thanks{This work was partially supported by the Swedish Research Council Grants 2020-03780 and 2024-04650.}



\dedicatory{In memory of Heinz Langer\\
who taught us not only mathematics}

\begin{document}

\maketitle

\begin{abstract}
The Colin de Verdi\`ere parameter is a number assigned to  discrete graphs which equals the maximal multiplicity of the second eigenvalue of a certain family of 
Laplacian matrices related to the graph. In this paper it is shown that the Colin de Verdi\`ere parameter can be obtained in the setting of metric graphs by looking 
at the maximal multiplicity of the second eigenvalue for Laplacians on metric graphs with delta couplings at the vertices. 
Two different families of Laplacians, as well as a family  of Schr\"odinger operators, all leading to the Colin de Verdi\`ere number are presented.
\end{abstract}

\section{Introduction}

The Colin de Verdi\`ere graph parameter among other things describes whether a 
discrete graph is planar or not. 
It generalises the classical result by Kuratowski \cite{Kuratowski}, but provides a new insight on the problem of a graph's planarity.
The parameter is defined as the maximal multiplicity of the second eigenvalue for a family of  certain Hermitian matrices associated with  the graph. 
With every discrete graph one may naturally associate a metric graph by assigning lengths to the edges.
We show that the Colin de Verdi\`ere 
parameter can also be described as the maximal multiplicity of the second eigenvalue for different families of 
Schr\"odinger operators acting 
on metric graphs. In constrast to discrete graphs, Schr\"odinger operators on metric graphs are determined by many parameters:
lengths of the edges, vertex conditions and potentials on the edges \cite{PK24,BeKu}.
In order to obtain a planarity criterion these parameters should satisfy certain additional assumptions, like the matrix $ \mathbf A $ in
Colin de Verdi\`ere original approach which should have negative entries outside the diagonal (see Section \ref{SecCdV}).
It is not  {\it a priori} clear what such additional assumptions should be  for metric graphs.
Moreover, optimisation problems become easier when the number of parameters is reduced; therefore we shall look for
minimal families first. These three alternative families are described in Section 4, 7, and 8.
These families are constructed in such a way that the corresponding secular matrix $ \mathbf Q(\lambda_2) $ (see (\ref{qq}))
just coincides with the Colin de Verdi\`ere matrix $ \mathbf A$. Suggested families of operators on metric graphs provide a nice geometric interpretation for the Colin de Verdi\`ere matrix,
especially 
 of the Strong Arnold Property (SAP), which is imposed on the  matrices
(see Section \ref{SecMG}).

Our studies concern differential operators on metric graphs, but our analysis provides a new insight
on the original Colin de Verdi\`ere construction. Let us first briefly mention this construction (for description see Section \ref{SecCdV}).

Given a (discrete) graph $ G =  (V,E) $ the Colin de Verdi\`ere parameter $\mu(G)$ is the maximal multiplicity of the second eigenvalue for a 
set of symmetric $|V| \times |V|$ matrices associated with the graph, called Colin de Verdi\`ere matrices (see Definition \ref{def1}). These matrices $\mathbf A$ are defined so
that their non-diagonal entries are strictly negative if there is an edge between the vertices and are equal to zero otherwise. Diagonal entries are arbitrary (real).
Moreover the matrices are required to have exactly one negative eigenvalue and possess the Strong Arnold Property (see Section \ref{SecCdV}).

Two fundamental ideas hidden behind the choice of the matrix $ \mathbf A $ make  this construction possible:
\begin{itemize}
\item the matrix $ \mathbf{A} $ associated with the graph $ G $ is chosen with nonpositive entries outside the diagonal 
in order to guarantee that the corresponding semigroup is positivity
preserving and Perron-Frobenius theorem holds, {\it i.e.} the ground state can be chosen positive  (see {\it e.g.} \cite{BerPle}
and \cite[Theorem 7.1]{Bat});
\item the proof of the planarity criterion is based on a sophisticated analysis of nodal domains corresponding to the eigenvectors associated with the second eigenvalue, which is possible only because the ground state is sign definite.
\end{itemize}

The Colin de Verdi\`ere parameter provides information about the planarity  of the discrete graph $G$ (see Theorem \ref{ThCdV}).

Let us discuss now how these ideas can be generalised for
differential operators on metric graphs.  
By metric graphs we understand finite collections of compact intervals (subsets of $ \mathbb R $)
with some sets of end points identified, thus forming vertices (see Section \ref{SecMG}).
Then the (self-adjoint)
Scrhr\"odinger operator $ - \frac{d^2}{dx^2} + q(x) $ acts on the functions defined on the edges
and satisfied certain conditions at the vertices
(see Section \ref{SecMG} where the double role of vertex conditions is explained).
Considered metric graph are compact, so the
spectrum of the operator is discrete and the eigenvalues tend to $ + \infty $.
Spectral theory of differential operators on
metric graphs can be seen as a natural generalisation of the spectral theory of discrete graphs allowing 
more rich structure of the spectrum.
 It is therefore tempting to investigate whether classical analysis of
Colin de Verdi\`ere can be  generalised to include metric graphs. 
It is clear that discussing such generalisation the
fundamental principles formulated above should be taken into account. 

Non-degeneracy and positivity of the ground state for
(unbounded) differential operators is guaranteed by the
positivity preserving property of the corresponding semigroups as described by classical
Beurling-Deny criterion \cite{BeDe,ReSi4}
This connection for Schr\"odinger operators on metric graphs
has been discussed in detail  in \cite{Ku19}, see also \cite[Section 4.5]{PK24}.
In particular it was shown that so-called generalised delta couplings at
the vertices guarantee positivity of the ground state independently of
the edge lengths. Therefore only vertex conditions given by generalised 
delta couplings will be considered throughout the paper (see \cite[Section 3.7]{PK24}).\footnote{Note that 
standard delta couplings at the vertices, and in particular standard vertex conditions, are special cases of 
the generalised delta couplings.}

In order to be able to speak about nodal domains the functions
at the vertices should satisfy weighted continuity conditions implying that
the class of generalised delta couplings should be reduced further
leading to weighted delta couplings described by Definition \ref{defgdc}.
The choice of vertex conditions is extremely important since as shown in Section \ref{SecDelta}
it is not sufficient to consider standard vertex conditions.

Let us discuss how minimal families of operators can be chosen.
 In the original works one varies a certain $ |V| \times |V| $ matrix $ \mathbf A$, whose non-diagonal entries 
are associated with the edges. 
The total number of parameters that are varied is 
\begin{equation} \label{parameters}
\# \; \mbox{vertices} + \# \; \mbox{edges}.
\end{equation}
With every discrete graph one may associate a Schr\"odinger operator on a metric graph. It appears that it is not enough to
consider standard Laplacians (see Section \ref{SecDelta}), therefore we consider the following three families, where the role
of vertex parameters is played by delta couplings:
\begin{enumerate}
\item Laplacians on metric graphs with delta couplings at the vertices (Section \ref{SecMG}).

\item Laplacians on equilateral metric graphs with weighted delta couplings at the vertices (Section \ref{SecWDC}).

\item Schr\"odinger operators with edge-wise constant potential
 on equilateral metric graphs with standard delta couplings at the vertices (Section \ref{SecSch}).
\end{enumerate}

The number of parameters in each family is given by formula \eqref{parameters}.
 The corresponding secular matrix given by \eqref{qq} for each family is a direct analog of the Colin de Verdi\`ere matrix $ \mathbf A$.
Geometric interpretation of the Colin de Verdi\`ere approach is especially clear when the first family if considered:
the parameters determining the matrix $ \mathbf A $ are just the lengths of the edges and the weights at the vertices.

 Our studies are closely related to studies of the spectral gap for Laplacians on metric graphs, comprehensive description of recent results for standard Laplacians can be found in 
 \cites{Nic1,Nic2,Fr,KuNa14,KeKuMaMu,BeKeKuMu23,BeKeKuMu17,BeKeKuMu19}, see also \cite[Chapter 12]{PK24}.

\section{Colin de Verdi{\`e}re parameter for discrete graphs} \label{SecCdV}

The Colin de Verdi{\`e}re parameter was originally defined in 1990 \cite{YC90} (see also translation \cite{YC93}). We are going to follow \cite{HV99}, which presented
a self-contained introduction into the subject.

Let $G=(V,E)$ be a connected discrete graph, where $V = \left\{ v^m \right\}_{m=1}^M $ is the set of vertices   and $ E $ is the set of edges.
We are going to write $ v^i \sim v^j $ if there is an edge between the vertices $ v^i $ and $ v^j$. Functions on $ G $ are defined on the vertices and can
therefore be identified with the vectors in $ \mathbb C^M$. We shall always assume that the graph is {\bf simple}, {\it i.e.} does not contain loops or
parallel edges.

With every discrete graph $ G $ 
let us associate  the set $\Sc_G\subset \sym_M$ of real-valued $ M \times M $ symmetric
 matrices $\mathbf{A}=(a_{ij})$, such that 
\begin{enumerate}
\item $a_{ii}$ is any real number,
\item $a_{ij} < 0$ if $ v^i\sim v^j$,
\item $a_{ij}=0$ if $v^i\not\sim v^j$ and $i\neq j$.
\end{enumerate}
We denote by 
$\Sc_G^\perp \subset \sym_M$ the set of matrices $ \mathbf X = (x_{ij}) $, such that
\begin{enumerate}
\item $a_{ii} = 0$,
\item $a_{ij} = 0$ if $ v^i\sim v^j$.
\end{enumerate}

\begin{remark}
If $ G $ is connected, then the ground state of any  $ \mathbf{A} \in \Sc_G $  is simple and the corresponding
eigenfunction can be chosen strictly positive. 
\end{remark}

Strict positivity comes from Perron-Frobenius theorem and is crucial for the choice of the set $ \Sc_G$.

We let $\xi _1 < \xi _2\leq ...$ denote the eigenvalues of $\mathbf{A}$. The idea behind the Colin de Verdi{\`e}re parameter is that one studies the multiplicity $ \mu $
of the second eigenvalue. Since the values on the diagonal are arbitrary one
 can always consider the matrix $\mathbf{A}-\xi _2\mathbf{I}$ instead of $\mathbf{A}$. Hence it is sufficient to consider matrices with precisely one negative eigenvalue and study the multiplicity of  $\xi _2=0$.
  It appears that the multiplicity of the second eigenvalue reflects topology of the graph only if it is in some sense stable under perturbations.

 One says that $\mathbf{A}\in \Sc_G$ has the {\bf Strong Arnold Property} if there is no non-zero symmetric matrix $\mathbf{X} \in \Sc_G^\perp  $  such that 
 \begin{equation} \label{eqSAP}
 \mathbf{AX}=\mathbf 0.
 \end{equation}
  The
 matrices in $\Sc_G$ with one negative eigenvalue, which in addition posses the Strong Arnold Property are called {\bf Colin de Verdi{\`e}re matrices}.

That $\mathbf{A}\in \Sc_G$ with $\dim( \ker(\mathbf{A}))=\mu$ has the Strong Arnold Property is equivalent to the fact that the spaces $\Sc_G$ and  
$$\Rc_\mu^M=\{\mathbf{A}\in \sym_M| \dim( \ker(\mathbf{A}))=\mu\}$$
intersect transversally at $\mathbf{A}$ in the space $\sym_M$.

\begin{definition}  \label{def1}
The {\bf Colin de Verdi{\`e}re graph parameter} $\mu(G)$ equals the maximal multiplicity $\mu$ of the eigenvalue $\xi _2=0$ for Colin de Verdi{\`e}re matrices
associated with the graph $ G$.
\end{definition}

\begin{remark}
The optimal Colin de Verdi\`ere matrix is not unique: multiplication by a real parameter gives a family of optimal matrices. 
\end{remark}

Note that in the definition above only symmetric matrices $ \mathbf{A} $ possessing the Strong Arnold Property are considered, which
 gives the Colin de Verdi{\`e}re graph parameter the important feature of minor monotonicity.

\begin{theorem}[Theorems 2.1 and 2.2 in \cite{YC90} or Statement 1.3 in \cite{HV99}]
The Colin de Verdi{\`e}re graph parameter $ \mu (G) $ is minor monotone
$$ \mbox{$H$ is a minor of $G$} \quad \Rightarrow \quad \mu(H) \leq \mu(G). $$
In particular $ \mu(G)$ can not increase when edges are deleted or contracted.
\end{theorem}

The importance of the Colin de Verdi{\`e}re graph parameter is reflected in the fact that  it  determines whether discrete graphs are planar or not.

\begin{theorem}[Theorems 3.2 and 5.7 in \cite{YC90} and Theorem 3 in \cite{LL98}] \label{ThCdV}
Let $ G $ be a simple connected discrete graph and $ \mu $ denotes its Colin de Verdi{\`e}re parameter. Then it holds:
\begin{itemize}
\item $\mu(G)=1$ iff $G$ is a path,
\item $\mu(G)\leq 2$ iff $G$ is outer planar,
\item $\mu(G)\leq  3$ iff $G$ is planar,
\item $\mu(G)\leq  4$ iff $G$ is linklessly embeddable.
\end{itemize}

\end{theorem}
Properties $(1)-(3)$ were shown in Colin de Verdi\`ere's original article \cite{YC90} where  $\mu(G)$ was introduced for the first time,
the proof of
property $(3)$ naturally required use of Kuratowski's theorem \cite{Kuratowski}. The forth statement was proven in \cite{LL98}.

It has also been shown that for important classes of graphs it is not necessary to require additionally  the Strong Arnold Property
-- it is automatically fulfilled for matrices satisfying the other requirements \cite{vdH10,AS17}. 

\section{Metric graphs and M matrices} \label{Sec3}

 A {\bf metric graph} $\Gamma = (\mathbf E, \mathbf V) $ is given by a set $\mathbf{E}$ of intervals $E_n=[x_{2n-1},x_{2n}]$, $n=1,...,N$, called {\bf edges}, 
 and a partition $\mathbf{V}$ of the end points $x_1,...,x_{2N}$ into non empty equivalence classes  $(v^i)_{i=1}^M$, called {\bf vertices}. $\Gamma$ becomes a metric space 
 if  all the points in each vertex are identified. We write $v^i \sim v^j$ if there is an edge with one end point in $v^i$ and one in $v^j$. We assume that 
 the metric graphs are simple and connected. If $v^i \sim v^j$ we let $\ell_{ij}$ denote the length of the edge between $v^i$ and $v^j$.
  The corresponding edge will be denoted by $ E_{ij}$, in other words we shall use two notations for the edges:  enumerating them in a certain order
  as $ E_n, \; n= 1,2, \dots, N$ and indicating which vertices they connect as $ E_{ij}, \; i,j = 1,2, \dots, M.$ Since the graphs are simple the two parameters
  $ n $ and $ ij $ uniquely determine each other. 
  The corresponding discrete graph (sharing the same set of vertices and edges) will be denoted by $ G = G(\Gamma).$

The space of functions on $\Gamma$ is the space $L^2(\Gamma)=\bigoplus_{n=1}^NL^2(E_n)$. 
Only  Laplacians on $ \Gamma $ defined by the differential expression
$ L u = - \frac{d^2}{dx^2} u $
will be considered. The maximal operator is defined on the domain $ W_2^2 (\Gamma) = \bigoplus_{n=1}^N W_2^2 (E_n) $, but this operator is not symmetric.
To make it self-adjoint it is necessary to introduce certain vertex conditions
connecting values of the function and its first derivatives separately at each vertex.  To this end we introduce
the limiting values of the functions and their oriented first derivatives at the end points:
$
u(x_j) := \lim_{x \rightarrow x_j} u(x), \qquad \partial u(x_j) := (-1)^{j+1}  \lim_{x \rightarrow x_j}  u'(x),
$
where the limit is taken from inside the edges.

In the first part of the paper we shall only use delta couplings at the vertices, which we call  standard
in order to distinguish them from weighted delta couplings introduced later.

\begin{definition}
Let  $ \alpha^j  \in \mathbb R $ be a real coupling parameter. Then  
{\bf standard delta coupling} at the vertex $ v^j $ is given by the following conditions:
\begin{equation} \label{sdc}
\left\{
\begin{array}{l}
\vu (x_i) = \vu (x_l),=: \vu (v^j) \; x_i, x_l \in v^j, \quad   \\[3mm]
\sum_{x_i \in v^j} \partial \vu (x_i) = \alpha^j  \vu (v^j). 
\end{array} \right.
\end{equation}
\end{definition}

Note that the first condition implies that the function is continuous at the vertices and therefore
its values $ u (v^j) $ are well-defined. We denote by $L^{\vec \alpha}(\Gamma)$ the Laplacian on $\Gamma$ with standard delta couplings given by $\vec \alpha$.

The spectra of these operators can be described using M-functions (see Section 5.3.4 in \cite{PK24}). 
Let $ u $ be any solution to the differential equation
$ -u''=\lambda  u, \quad \im \lambda \neq 0 $
on each edge. Then the M-function connects together Dirichlet and Neumann data at each vertex:
\begin{equation}  \label{eqM}
 \small \mathbf M (\lambda):
\left(
\begin{array}{c}
u (v^1)  \\
u (v^2)  \\
\vdots \\
u (v^M) 
\end{array}
\right)  \mapsto \left(
\begin{array}{c}
\partial u (v^1) \\
\partial u (v^2)  \\
\vdots \\
\partial u (v^M) \end{array}
\right) ,
\end{equation}
where we use notation
$ \partial u (v^j) := \sum_{x_i \in v^j} \partial u (x_i). $

Let the edge between two vertices $ v^i$ and $ v^j $ with length $ \ell_{ij}$ be parametrized as $ E_1 = [x_1, x_2] $. Then
the Titchmarsh-Weyl M-function for the edge connecting the function values $ (u(x_1), u(x_2)) $ to the normal derivative values $ (\partial u(x_1), \partial u(x_2) ) $
is given by the $ 2 \times 2 $ matrix
\begin{equation} \label{7}
 \mathbf M_{\bf e} (\lambda) : = \begin{pmatrix} -\frac{\sqrt{\lambda }\cos(\sqrt{\lambda }\ell_{ij})}{\sin(\sqrt{\lambda }\ell_{ij}) } & \frac{\sqrt{\lambda }}{\sin(\sqrt{\lambda }\ell_{ij} )}
 \\ \frac{\sqrt{\lambda }}{\sin(\sqrt{\lambda }\ell_{ij} )} & -\frac{\sqrt{\lambda }\cos(\sqrt{\lambda }\ell_{ij} )}{\sin(\sqrt{\lambda }l\ell_{ij} )}\end{pmatrix}.
 \end{equation}

If this matrix has a non-trivial kernel for some $\lambda $,
 then that $\lambda $ is a Neumann eigenvalue for that edge since there must be some function satisfying $-u''=\lambda  u$ and $(u(x_1),u(x_2))\neq (0,0)$ but with $(\de u(x_1),\de u(x_2))= (0,0)$.

The matrix M-function for $ \Gamma $ is formed from the $ 2 \times 2 $ edge M-functions:
\begin{equation}  \label{edgeM}
\small
\left(\mathbf M (\lambda) \right)_{ij} =
\left\{
\begin{array}{ll}
- \displaystyle \sum_{v^m \sim v^i} \sqrt{\lambda} \cot  \sqrt{\lambda}  \ell_{im}, & j =i; \\[3mm]
 \displaystyle \frac{ \sqrt{\lambda} }{\sin  \sqrt{\lambda}  \ell_{ij}} , & v^i \sim v^j; \\[3mm]
\displaystyle  0, &  v^i \not\sim v^j,j \neq i. \\[3mm]
 \end{array} \right.
\end{equation}

We note that $\lambda \mapsto\cos(\sqrt{\lambda }\ell_{ij})$ is an entire function and $\lambda \mapsto\frac{\sqrt{\lambda }}{\sin(\sqrt{\lambda }\ell_{ij})}$ is 
a meromorphic function with simple poles at $(\frac{\pi}{\ell_{ij}})^2n^2$, where $n$ is a positive integer.
Hence the matrix valued M-function can be continued to the real axis for all $ \lambda \neq \left( \frac{\pi}{\ell_{ij}} \right)^2n^2, \; n  = 1,2, \dots$.

A real number $ \lambda \not= \left(\frac{\pi}{\ell_{ij}}\right)^2n^2, \; \forall \ell_{ij}  ,$ is an eigenvalue
 of the Laplacian with delta interactions only if the
 linear equation \begin{equation} \label{eq7}
\small \Big( \mathbf M(\lambda) - {\rm diag}\; \left\{ \alpha_{m} \right\} \Big) 
\left(
\begin{array}{c}
 u(v^1) \\
\vdots \\
u(v^M)
\end{array}
\right)
= 0  
\end{equation}
has a non-trivial solution. We define the secular matrix
\begin{equation} \label{qq} 
\mathbf Q(\lambda):= \mathbf M(\lambda) - {\rm diag}\; \left\{ \alpha_{m} \right\}.
\end{equation}
The multiplicity of the eigenvalue coincides with the dimension of the kernel of $\mathbf Q(\lambda)$, hence we have proven:

\begin{lemma} \label{lem1}
The dimension of $\ker( \mathbf Q(\lambda))$, for $\lambda$ such that $\lambda\neq(\frac{\pi}{\ell_{ij}})^2$ for all $i\sim j$, equals the multiplicity of the eigenvalue $\lambda $ for the associated Laplacian $L^{\vec \alpha}(\Gamma)$.
\end{lemma}

To understand relation between the singularities of the M-function and eigenvalues of $ \mathbf Q (\lambda) $ we use ideas developed in \cite{KuNa}. 
The M-function is a Matrix valued Herglotz-Nevanlinna function i.e. for $\lambda\in \mathbb{C}\backslash\mathbb{R}$ we have that $M(\lambda)$ is analytic, 
symmetric $\mathbf M(\overline{\lambda})=\mathbf M(\lambda)^*$ and 
has non-negative imaginary part in the upper half-plane  $\frac{\im{\mathbf M(\lambda)}}{\im{\lambda}}\geq 0$ where $\im(\mathbf M(\lambda))=\frac{\mathbf M(\lambda)-\mathbf M(\lambda)^*}{2i}$.

\begin{lemma} \label{lem2}
The Titchmarsh-Weyl M-function $ \mathbf M(\lambda) $ is a Herglotz-Nevanlinna $ M \times M $ matrix-valued function with the following properties:
\begin{itemize}
\item For every $ \lambda \in \mathbb R $, that is not a singularity $(\frac{\pi}{\ell_{ij}})^2$ of $\mathbf M(\lambda) $, there are precisely $ M $ real eigenvalues  $ \xi_j (\lambda), \, j =1,2, \dots, M $, depending analytically on $ \lambda $
\footnote{If two eigenvalue curves cross each other, then to keep analyticity, one may need to adjust the order of the indexes.}
\item The eigenvalues are increasing as functions of real $\lambda$ except for possible singularities at the points 
\begin{equation} \label{sing}
\lambda=\left(\frac{\pi}{\ell_{ij}}\right)^2 n^2,  \quad n \in \mathbb N.
\end{equation}
\item All $M$ eigenvalues tend to $-\infty$ as $\lambda$ tends to $-\infty$
satisfying the asymptotics:
\begin{equation} \label{asev}
\xi_j \sim - d_j \sqrt{|\lambda|}, \quad \lambda \rightarrow - \infty,
\end{equation}
where $ d_j $ is the degree of the vertex $ v^j$.
\end{itemize}
\end{lemma}

\begin{proof}
Explicit formula \eqref{7}  for  $\mathbf{M}(\lambda)$ implies that this matrix-valued function is analytic for non-real lambda and that $\mathbf{M}(\lambda)^*=\mathbf{M}(\overline{\lambda})$.
The singularities of the M-function coincide with the zeroes of sine functions appearing in the denominator -- the Dirichlet-Dirichlet
eigenvalues of the Laplacian on the edges. These singularities are given by formula \eqref{sing}.

The $ M \times M $ matrix-valued function $\mathbf{M}(\lambda)$  is a sum of $ 2 \times 2 $ matrix-valued  M-functions for single edges 
given by formula \eqref{edgeM}. Each of these matrix-valued functions is a Herglotz-Nevanlinna function. This is since for $\im(\lambda)>0$ the eigenvalues of 
$$\im(\mathbf{M}_e(\lambda))
= 
\begin{pmatrix}
\im\left(-\frac{\sqrt{\lambda }\cos(\sqrt{\lambda }l_{ij})}{\sin(\sqrt{\lambda }l_{ij}) }\right) &
\im\left(\frac{\sqrt{\lambda }}{\sin(\sqrt{\lambda }l_{ij}) }\right)\\[3mm]
\im\left(\frac{\sqrt{\lambda }}{\sin(\sqrt{\lambda }l_{ij}) }\right) &
\im\left(-\frac{\sqrt{\lambda }\cos(\sqrt{\lambda }l_{ij})}{\sin(\sqrt{\lambda }l_{ij}) }\right)
\end{pmatrix}
$$ are 
$$\im\left(-\frac{\sqrt{\lambda }\cos(\sqrt{\lambda }l_{ij})}{\sin(\sqrt{\lambda }l_{ij}) }\right)\pm \im\left(\frac{\sqrt{\lambda }}{\sin(\sqrt{\lambda }l_{ij}) }\right). $$
If we let $\sqrt{\lambda}l_{ij}=:x+iy$ then the eigenvalues can be written as 
$$\frac{1}{\ell_{ij}}\im\left((x+iy)\tan(\frac{x+iy}{2})\right), -\frac{1}{\ell_{ij}}\im\left((x+iy)\cot(\frac{x+iy}{2})\right)$$
where $x,y>0$. $\im(\mathbf{M}_e)$ is then positive definite since we have that 
$$\im\left((x+iy)\tan(\frac{x+iy}{2})\right)=\frac{\sinh(y)x+\sin(x)y}{\cosh(y)+\cos(x)}>0$$
and
$$-\im\left((x+iy)\cot(\frac{x+iy}{2})\right)=\frac{\sinh(y)x-\sin(x)y}{\cosh(y)-\cos(x)}>0.$$
 Hence $\mathbf{M}(\lambda)$ is also a Herglotz-Nevanlinna function. 
 
 For real $ \lambda $ different from \eqref{sing}  the matrix $\mathbf{M}(\lambda) $ is Hermitian, hence perturbation theory  (see {\it e.g.} \cite[Theorem 6.1]{Kato}) implies
 that the eigenvalues of $ \mathbf M(\lambda) $ depend analytically on $ \lambda $. They are non-decreasing functions since the M-function is Herglotz-Nevanlinna.

 As $\lambda$ tends to $-\infty$ all the elements of $\mathbf{M}(\lambda)$ outside of the diagonal tend to $0$.
 On the diagonal $ \cot \sqrt{\lambda} l_{im} $ tends to $1$, which gives 
 asymptotics \eqref{asev}  taking into account that for every vertex there are precisely $ d_j $ terms  of the form
 $  - \sqrt{\lambda} \cot \sqrt{\lambda} \ell_{im}  \sim - \sqrt{\lambda}$.
\end{proof}

Note that precisely the same conclusions as in Lemma \ref{lem2} hold for the Herglotz-Nevanlinna matrix valued function
$\mathbf Q(\lambda) := \mathbf M (\lambda) - {\rm diag}\, \{ \alpha_j \}. $
For small negative $ \lambda \rightarrow - \infty$ the matrix $ \mathbf Q(\lambda) $ is strictly negative. As $ \lambda $ increases, the eigenvalues
$ \xi_j (\lambda) $ of $ \mathbf Q(\lambda) $  increase as well and the ground state for the graph corresponds to the $ \lambda_1 $ for which the largest
$ \xi_j $, say $ \xi_1 (\lambda) $ crosses the line $ \xi = 0$, {\it i.e.}
$$ \xi_1(\lambda_1) = 0. $$
The ground state of $ \Gamma$ is simple \cite{PK24,Ku19}, hence no other energy curve $ \xi_j (\lambda), j \neq 1 $ crosses the line at this value of $ \lambda$.

Directly to the right of the point $ \lambda_1 $, the matrix-valued function $ \mathbf Q(\lambda) $ has precisely one positive eigenvalue.
Two scenarios may occur:
\begin{itemize}
\item The second eigenvalue of the Laplacian on $ \Gamma $ corresponds to the $ \lambda_2 $ for which the second
energy curve approaches the line $ \xi = 0 $.
\item The first energy curve approaches the singularity and goes to $ + \infty $ and continues to the right of the singularity starting from $- \infty $.
In this case $ \lambda_2 $ appears again as the point, where the largest $ \xi_j$ crosses the axis $ \xi = 0$.
\end{itemize}

In the first case the matrix $ \mathbf M (\lambda_2) $ has $ \xi= 0 $ as the second largest eigenvalue.
In the second case $ \mathbf M (\lambda_2) $ is a non-positive matrix with $ \xi= 0 $ as the largest eigenvalue.

There is also a border case where $ \lambda_2 $ coincides with the lowest singularity of the M-function.
In what follows we shall only consider metric graphs corresponding to the first case, where $ \lambda_2 $ is lying below
the lowest singularity.
The multiplicity of $ \lambda_2 $ coincides with
the number of energy curves crossing the axis at this energy.

\section{Metric graphs with standard delta couplings} \label{SecMG}

In this section we shall consider Laplacians on metric graphs with arbitrary edge lengths and standard delta conditions at the vertices. Hence
we shall always assume that the functions from the domain are continuous at the vertices. The parameters that will vary will include
the lengths of the edges $ \ell_{ij} \in \mathbb R_+$,
and the delta couplings at the vertices $ \alpha_j \in \mathbb R$.

Our strategy will be to show that for any Colin de Verdi\`ere matrix $\mathbf{A}$ we can find a metric graph $ \Gamma $ associated with the discrete graph $ G $ so that
 equation 
 \begin{equation} \label{eq2}
\mathbf M (\lambda_2) - {\rm diag}\; \{ \alpha_j \} = - \mathbf A.
\end{equation}
holds for the second eigenvalue $ \lambda_2 (\Gamma) $. In addition we need
that $ \lambda_2 (\Gamma) $ lies below all singularities of the M-function.
To this end we introduce the family of metric graphs with delta couplings at the vertices,
which we call admissible.

\begin{definition}
A metric graph $ \Gamma $ with delta couplings $ \alpha_j $ at the vertices is called
{\bf admissible} if and only if:
\begin{itemize}
\item the second eigenvalue $ \lambda_2 (\Gamma) $ lies below all singularities 
of the M-function. Since the lowest singularity is $\left( \frac{\pi}{\ell_{\rm max}} \right)^2$ this means that
\begin{equation} \label{eq9}
\lambda_2 (\Gamma) <  \left( \frac{\pi}{\ell_{\rm max}} \right)^2; 
\end{equation}
\item the matrix $ \mathbf Q(\lambda_2) = \mathbf M (\lambda_2 (\Gamma)) - {\rm diag}\, \{\alpha_j \} $
possesses the Strong Arnold Property.
\end{itemize}
\end{definition}

Here $ \ell_{\rm max} $ is the length of the longest edge:
$ \ell_{\rm max} = \max_{i \neq j}  \ell_{ij}. $

Condition \eqref{eq9} is required 
in order to guarantee that the value of the eigenfunction on an edge is
uniquely determined by its values at the end points. In general we have only the estimate $ \lambda_2 (\Gamma) \leq \left( \frac{2 \pi}{\ell_{\rm max}} \right)^2 $
valid for Laplacians with any vertex conditions. Condition \eqref{eq9} implies that the non-diagonal elements of the matrix $ \mathbf Q(\lambda_2) $ are nonnegative.

\subsubsection*{On Strong Arnold Property for metric graphs}  
The Strong Arnold Property (SAP) is crucial for selecting matrices appearing in the Colin de Verdi\`ere criterium and it has already been
discussed for discrete Laplacians  in Section \ref{SecCdV}. This property has a remarkable interpretation for metric graphs
connected with the zero sets of the eigenfunctions.

The zero sets of eigenfunctions of differential operators on metric graphs may have non-zero measure:
such functions vanish on whole edges. This property distinguishes
differential operators on metric graphs from conventional differential operators on smooth manifolds. For example, if the metric graph has a loop,
then there is a series of eigenfunctions supported by this loop and vanishing on the rest of the graph.
More generally, solution of the eigenfunction equation on each edge is determined by two constants
(as a solution to the second order differential equation), hence if an eigenvalue has multiplicity at least three, then 
 one may always find an eigenfunction vanishing on any  edge in the graph. 
Under condition \eqref{eq9} vanishing of the eigenfunction at any two neighbouring vertices guarantees
that the eigenfunction is identically zero on the edge between these two vertices.

With any vertex $ v^0 $ in the metric graph $ \Gamma $ one naturally associates the (metric) star graph formed
by $ v^0 $ and all adjacent edges together with adjacent vertices. Every graph $ \Gamma $ has precisely $ M$
such star subgraphs, which we consider as embedded into the original metric graph. Some of these star subgraphs degenerate into 
single intervals in the case of degree one vertices.
To guarantee that adjacent to $ v^0$ edges form a star subgraph, it is necessary and sufficient to assume 
that the original graph is simple (no loops or multiple edges), otherwise the adjacent to $ v^0$ edges may form loops or cycles in $ \Gamma$.

Consider for example the star graph of degree $6$ (see Fig. \ref{FigStar}). This graph has $7$ vertices and hence $7$ star-subgraphs: 
\begin{itemize}
\item the graph $ S_6 $ itself corresponding to the central vertex $ v^5$;
\item $6$ single edge graphs corresponding to the degree one vertices $ v^1, v^2, v^3, v^4. v^5, $ and $v^6$.
\end{itemize}

\begin{figure}[ht]
\centering
\begin{tikzpicture}[scale=0.5,every node/.style={circle,fill=black,inner sep=1.5pt}]

\node[label=right:1] (1) at (2,0){}; 
\node[label=above:2] (2) at (1,1.73){}; 
\node[label=above:3] (3) at (-1,1.73){}; 
\node[label=left:4] (4) at (-2,0){}; 
\node[label=below:2] (5) at (-1,-1.73){}; 
\node[label=below:1] (6) at (1,-1.73){}; 

\node (7) at (0,0) {}; 

\put (-2,-10){$7$}

\draw [thick] (7) to (1);
\draw [thick] (7) to (2);
\draw [thick] (7) to (3);
\draw [thick] (7) to (4);
\draw [thick] (7) to (5);
\draw [thick] (7) to (6);

\end{tikzpicture}
\caption{The star graph $ S_6$.}
\label{FigStar}
\end{figure}
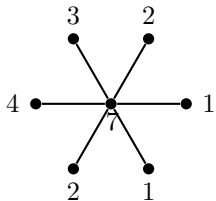
\vspace{-2mm}

Let us show that if the graph is equilateral, say has  edge lengths $1$, then the maximal multiplicity of the second eigenvalue
can be equal to $5$.  Let us assume standard vertex condition
at $ v^7 $ and  delta couplings $ \partial u (v^j) = -  u (v^j), \, j =1,2, \dots, 6, $ att degree one vertices (see Fig, \ref{FigStar}).
The eigenfunctions corresponding to $ \lambda_2 = 0 $ are  linear on the edges, vanish
at $ v^7$ and satisfy delta couplings at other vertices.
Parameterizing the edges as $[0,1]$ in the direction away from the central vertex  the eigenfunction on the edges are proporltonal to   $ x$.
To meet standard conditions at $ v^7 $ the eigenfunction should be supported
by at least two edges. We get that  multiplicity of $ \lambda_2 = 0 $ equals $5$, despite that the graph is obviously planar.
The reason is that this metric graph does not satisfy SAP:  among the eigenfunctions corresponding to $ \lambda  = 0 $
one finds functions vanishing on any of the edges, which are (trivial) star subgraphs for $ S_6$.

Determining planarity of graphs vertices of degree one may be ignored: one may always delete these
vertices together with the adjacent edges, repeating, if necessary, the procedure until it stabilizes. One may also remove all degree
two vertices substituting the corresponding two adjacent edges with one single edge.
Assume  from now on that
the graph $ \Gamma $ does not have any  vertex of degree one or two,  therefore all star subgraphs contain at 
least three edges. Multiple eigenvalues (as explained above) always lead to eigenfunctions vanishing on certain edges,
but vanishing on star subgraphs of degree at least $3$ is much more rare.
For example multiplicity $3$ of the eigenvalue guarantees existence of eigenfunctions vanishing on edges, while
existence of eigenfunctions vanishing on $3$-star subgraphs can be guaranteed only if the multiplicity is $4$.

Violation of SAP is related to the existence of eigenfunctions vanishing on the described star subgraphs.

\begin{lemma}  \label{Lemma42}
Let $ \Gamma $ be s simple metric graph without vertices of degree neither one nor two, and
let $ L^{\vec{\alpha}}(\Gamma) $ be the Laplacian on $ \Gamma $ with delta couplings
at the vertices. Let $ \lambda^*$ satisfying
\begin{equation} \label{lest}
 \lambda^* < \left( \frac{\pi}{\ell_{\rm max} (\Gamma)} \right)^2 
 \end{equation}
  be an eigenvalue of $ L^{\vec{\alpha}}(\Gamma) $,
such that there is no two eigenfunctions vanishing on some two star subgraphs of $ \Gamma $, 
then the secular matrix $ \mathbf Q(\lambda^*) $ possesses SAP.
\end{lemma}

\begin{proof}
The secular matrix $ \mathbf Q(\lambda^*) $ -- analog of the matrix $ \mathbf A $, -- does {\bf not} possess SAP only if there exists a
matrix $ \mathbf X \in \Sc_G^\perp ,$ such that $ \mathbf Q(\lambda^*) \mathbf X = 0 .$
Every column in $ \mathbf X $ belongs to the kernel of $ \mathbf Q(\lambda^*) $ and therefore
determines an eigenfunction of the Laplacian on the metric graph for the eigenvalue $ \lambda^*$.
Consider any such eigenfunction, say associated with the first column. In this column all entries corresponding to
the vertices to which $ v^1 $ is adjacent are equal to zero since $ \mathbf X  \in \Sc_G^\perp$. Then inequality 
\eqref{lest} guarantees that the corresponding eigenfunction
on the metric graph is identically equal to zero on the whole star subgraph associated with  $ v^1 $.
The matrix $ \mathbf X $ should be symmetric, hence it can be chosen non-zero only if
at least two metric graph eigenfunctions vanish on two (different) star subgraphs.
\end{proof}

The above Lemma can be strengthened, by requiring that there are no eigenfunctions vanishing on three different star subgraphs of
$ \Gamma $. To prove this one needs to show that any eigenfunction is non-zero on at least three vertices in $ \Gamma$. This follows
from the fact that the graph is simple and has no degree one and two vertices. Of course, inequality \eqref{lest}
should be taken into account.

This lemma may hold only if the multiplicity of the eigenvalue $ \lambda^*$
does not exceed the second lowest vertex degree in $ \Gamma $ by $1$.
For example if the multiplicity of $ \lambda^* $ is at least $4$, then for any vertex of degree $3$, one may always obtain 
eigenfunctions vanishing on the corresponding star graph.

We are going to consider only metric graphs associated with a fixed discrete graph $ G$. As we already mentioned in
the Introduction for certain discrete graphs the second condition in the definition above
is automatically satisfied.

Let us generalise the Colin de Verdi\`ere parameter  for metric graphs as follows:
\begin{definition}
The Colin de Verdi\` ere parameter $ \mu (\Gamma) $ for metric graphs is equal to
the maximal multiplicity of $\lambda_2 (\Gamma) $ for Laplacians  $ L_{\vec{\alpha} \delta} (\Gamma) $ with standard delta couplings at the vertices when one varies
over all admissible metric graphs associated with the same discrete graph.
\end{definition}

This graph parameter can be used to determine planarity of the graphs, similar to the
classical Colin de Verdi\` ere number.

\begin{theorem}  \label{ThST} The Colin de Verdi{\`e}re graph parameters $\mu(G)$ and $ \mu (\Gamma) $ are equal, 
provided $ G $ is the discrete graph associated with the metric graph $ \Gamma$.
\end{theorem}

\begin{proof}
Our goal is to prove that $ \mu(G) = \mu (\Gamma) $. We divide the proof into two steps:
\begin{enumerate}
\item $ \mu (G) \leq \mu (\Gamma) .$
Let $\mathbf{A}$ be a Colin de Verdi{\`e}re matrix with maximal multiplicity of the second eigenvalue. We pick a positive number $\lambda_2$ such that $\sqrt{\lambda_2}<|a_{ij}|$ for all $v^i\sim v^j$ and then we 
construct an operator $L_{\vec{\alpha} \delta}$ with $\lambda_2$ as its second eigenvalue and equation \eqref{eq2} is satisfied.

Consider any edge, say between two vertices $ v^i $ and $ v^j $. The corresponding entries of the matrices $ - \mathbf Q (\lambda_2) $ and $ \mathbf A $ are given by
$$ -\frac{\sqrt{\lambda_2}}{\sin(\sqrt{\lambda_2} \ell_{ij})} \quad \mbox{and} \quad a_{ij} < 0,$$
respectively.
The function $ \ell \mapsto \frac{\sqrt{\lambda_2}}{\sin(\sqrt{\lambda_2} \ell )}$ with  the domain $\ell \in (0,\frac{\pi}{\sqrt{\lambda_2}})$ has the range $[\sqrt{\lambda_2},\infty)$, and  
hence there is $ \ell_{ij} $ making the entries equal
\begin{equation} \label{equal}
a_{ij} = -\frac{\sqrt{\lambda_2}}{\sin(\sqrt{\lambda_2} \ell_{ij})} ,
\end{equation}
 since we have chosen $ \lambda_2 $ satisfying $\sqrt{\lambda_2}<|a_{ij}|$.
Note that there are two possible values of $ \ell_{ij} $, which can be chosen arbitrarily.

In this way we get a metric graph $ \Gamma $ so that
the matrices $-\mathbf{Q}(\lambda_2)$ and $\mathbf{A}$ have identical entries outside the diagonal. 
To make the two matrices identical $-\mathbf{Q}(\lambda_2)=\mathbf{A}$
we adjust 
the values of  the $\delta$-couplings in $ L_{\vec{\alpha}\delta}$.
Since all singularities of $ \mathbf{Q}$ are to the right of $ \lambda_2$
(condition \eqref{eq9}) and $\xi=0$ is the second lowest eigenvalue of $\mathbf{A}$, $\lambda_2$ is the second eigenvalue of $ L_{\vec{\alpha} \delta}$. $ \mathbf{Q}(\lambda_2)$ possesses the Strong Arnold 
Property since $\mathbf{A}$ does. This accomplishes the proof of the first part.

\item $ \mu (G) \geq \mu(\Gamma) .$ Consider one of the operators $ L_{\vec{\alpha}\delta}$, which maximises 
the multiplicity of $ \lambda_2 $ among admissible metric graphs with standard delta couplings. Then taking
$ \mathbf{A} =  - \mathbf Q(\lambda_2) $
we obtain a Colin de Verdi{\`e}re matrix.
Equation \eqref{eq9}
 implies that $ \sin \sqrt{\lambda}_2 \ell_{ij} >0 $, which in turn implies that non-diagonal elements of $ -\mathbf Q(\lambda_2) $
 are negative for $ v^i \sim v^j$. Moreover $ \xi = 0 $ is the second smallest
eigenvalue of $ \mathbf{A}$. The second assertion follows.

\end{enumerate}

\end{proof}

The graph parameter $ \mu (\Gamma) $ is determined by the discrete graph associated with $ \Gamma$. The advantage 
of introducing this new parameter is the possibility to give geometric interpretation for the
entries of the Colin de Verdi\`ere matrix $ \mathbf A$.

Of course trying to generalise the Colin de Verdi\`ere parameter we first tried with standard Laplacians
on metric graphs (corresponding to coupling constants $ \alpha_j = 0$). It appears that this family of
metric graphs is not rich enough.

\section{Delta couplings are necessary} \label{SecDelta}

One may ask the following question: Is it possible to get a generalisation of Colin de Verdi{\`e}re criterion using just standard Laplacians without
delta interactions at the vertices. In what follows we present a counterexample showing that there is a non-planar graph with the highest
multiplicity of $ \lambda_2 $ less or equal to $3$.

Consider the graph $K_{3,3}$. Label the vertices in the first set by $1,2,3$ and in the second set by $4,5,6$. 
Form a new graph $\Gamma$ by adding an edge between the vertices $2$ and $3$. We want to show that the second eigenvalue has maximal multiplicity 3. 
By scaling the edges we may assume that $\lambda _2=1$. This implies that at most one edge can have length $\geq \pi$ and no edge can have length $\geq 2\pi$.
Otherwise $ \lambda_2 < 1 $ as there exists a trial function, which is not an eigenfunction, orthogonal to $ \psi_1 \equiv 1 $ and having Rayleigh quotient equal to $ 1$
(the proof follows the same lines as the proof of Lemma \ref{lemmaW}).
\vspace{-2mm}

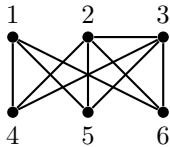
\begin{figure}[ht]
\centering
\begin{tikzpicture}[scale=0.5,every node/.style={circle,fill=black,inner sep=1.5pt}]

\node[label=above:1] (1) at (0,2){}; 
\node[label=above:2] (2) at (2,2){}; 
\node[label=above:3] (3) at (4,2){}; 
\node[label=below:4] (4) at (0,0) {}; 
\node[label=below:5] (5) at (2,0) {}; 
\node[label=below:6] (6) at (4,0) {}; 

\draw [thick] (1) to (4);
\draw [thick] (1) to (5);
\draw [thick] (1) to (6);
\draw [thick] (2) to (4);
\draw [thick] (2) to (5);
\draw [thick] (2) to (6);
\draw [thick] (3) to (4);
\draw [thick] (3) to (5);
\draw [thick] (3) to (6);

\draw [thick] (2) to (3);
\end{tikzpicture}
\caption{The graph $K_{3,3}$ with an added edge.}
\label{Fig2.3}
\end{figure}
\vspace{-2mm}

 If one edge has length $\pi$ then we shall check that the maximal multiplicity of $1$ is at most $3$, this will allow us to assume that no edge has length $\pi$. 
 We need to consider three possibilities:

\centerline{ (1) $ \ell_{23} = \pi$; \quad (2) $ \ell_{14} = \pi$; \quad (3) $ \ell_{36} = \pi$.}

 As illustration consider the first case. Assume for the sake of contradiction that the multiplicity of
$ \lambda_2 $ is at least $ 4$. Then there exists a not identically equal to zero eigenfunction equal to zero
at the vertices $ v^3,v^5,v^6$. Looking at the edge between $ v^2$ and $ v^3 $ we conclude that the function is equal
to zero also  at $ v^2$. The function must be identically equal to zero on the edges $ E_{25}, E_{35}, E_{36}, E_{26} $ since their
lengths are less than  $ \pi$. Consider the vertex $ v^5 $. All except one normal derivatives
are zero there, so the eigenfunction on the edge $E_{15}$ satisfies trivial Cauchy data
at the end point belonging to $ v^5$. Hence the eigenfunction is identically equal to zero on the edge $E_{15}$.
Similar reasoning implies that the eigenfunction is zero on the edge $ E_{16}$. Considering the vertex $ v^1 $ we conclude
that the eigenfunction is zero on the edge $E_{14}$, in particular it is zero at the vertex $ v^4$.
The edges $E_{24}$ and $ E_{34} $ have length less than $ \pi $, hence the eigenfunction is zero on these edges.
We conclude finally that the eigenfunction must be equal to zero on the edge $E_{23}$
and therefore is trivial. Our assumption was wrong.

 Consider now the general case $ \ell_{ij} \neq \pi$.  The M-function is of the form
$$ \small \begin{pmatrix}
m_{11} & 0  & 0 & m_{14} &m_{15} & m_{16}\\
0 & m_{22}  & m_{23} & m_{24} &m_{25} & m_{26}\\
0 & m_{23}  & m_{33} & m_{34} &m_{35} & m_{36}\\
m_{14} &m_{24} & m_{34} &m_{44}&0&0\\
m_{15} &m_{25} & m_{35} &0&m_{55}&0\\
m_{16} &m_{26} & m_{36}&0&0 &m_{66}\\
\end{pmatrix} =: \begin{pmatrix}\vec{m}_1 & \vec{m}_2& \vec{m}_3&\vec{m}_4 & \vec{m}_5& \vec{m}_6\end{pmatrix}
$$
and all entries are finite.\footnote{The form of the M-function with several entries equal to zero is not enough to prove that the multiplicity
 of the zero eigenvalue cannot be equal to $4$. Counterexample is given by  the matrix
 $$ \tiny \begin{pmatrix}
-1 & 0  & 0 & 1 & 1 & 1\\
0 & 1  & 1 & 1 & 1 & 1\\
0 & 1  & 1 & 1 & 1 & 1\\
1 &1 & 1 & 0&0&0\\
1 &1 & 1 & 0&0&0\\
1 &1 & 1 & 0&0&0\\
\end{pmatrix} .$$ 
}

To get multiplicity 4 only two columns of this matrix can be independent. Clearly $\vec{m}_1$ and $\vec{m}_4$ are independent, since $ m_{ij} = \frac{1}{\sin \ell_{ij}} \neq 0, \; i \neq j $.
We see that $\vec{m}_5,\vec{m}_6$ must be multiples of $\vec{m}_4$. 
Thus it is necessary that $m_{44}=m_{55}=m_{66}=0$. Remember that only one edge can be longer than $ \pi $. Hence only one $ m_{ij}, i \neq j , $ in question can be negative,
so there are $a,b>0$ such that $a\vec{m}_4=\vec{m}_5$, $b\vec{m}_4=\vec{m}_6$.

Since $m_{44}=0$ we have that
\begin{equation} \label{cot}
\cot(\ell_{14})=-(\cot(\ell_{24})+\cot(\ell_{34})).
\end{equation}
We note that for $i=1,2,3$
\begin{equation} \label{eq14}
\frac{1}{\sin(\ell_{i5})}=\frac{a}{\sin(\ell_{i4})}=a\sqrt{1+\cot^2(\ell_{i4})}
\end{equation} and since $m_{55}=0$ we get
$$0=\pm\sqrt{a^2(1+\cot^2(\ell_{14}))-1}\pm\sqrt{a^2(1+\cot^2(\ell_{24}))-1}\pm\sqrt{a^2(1+\cot^2(\ell_{34}))-1}$$
With equation \eqref{cot} the possible solutions are $a^2=1$ and 
$$a^2=\frac{3}{3+2(\cot^2(\ell_{14})+\cot^2(\ell_{24})+\cot^2(\ell_{34}))}\leq 1$$
but $a^2=1$ is the only solution that leads to real valued lengths.
This is because equation \eqref{eq14} leads to
$$ 
\frac{1}{\sin(\ell_{i5})} =  \frac{ \sqrt{1+\cot^2(\ell_{i4})}}{\sqrt{1+ 2\frac{\cot^2(\ell_{14})+ \cot^2(\ell_{24}) + \cot^2(\ell_{34})}{3}}} . 
$$
If we do not have $\cot^2(\ell_{14})=\cot^2(\ell_{24})=\cot^2(\ell_{34})=0$ then for at least one $ i $ we get  $ \frac{1}{\sin(\ell_{i5})}  < 1$, which is impossible. We conclude that $a^2=1$ is the only solution{, so $a=1$. By the same argument $b=1$.

This shows that $\vec{m}_4=\vec{m}_5=\vec{m}_6$. One can see that if $\vec{m}_2$ and $\vec{m}_3$ are going to be linear combinations of $\vec{m}_1,\vec{m}_4$ then we need to have 
$\frac{m_{11}}{m_{14}^2}=\frac{-m_{22}}{m_{24}^2}=\frac{-m_{33}}{m_{34}^2}.$
This is equivalent to
$$3\frac{\cot(\ell_{24})+\cot(\ell_{34})}{1+(\cot(\ell_{24})+\cot(\ell_{34}))^2}=\frac{\cot(\ell_{23})+3\cot(\ell_{24})}{1+\cot(\ell_{24})^2}=\frac{\cot(\ell_{23})+3\cot(\ell_{34})}{1+\cot(\ell_{34})^2}.$$
The last equation gives us that $\cot(\ell_{23})=3\frac{1-\cot(\ell_{24})\cot(\ell_{34})}{\cot(\ell_{24})+\cot(\ell_{34})}$ so 
$$3\frac{\cot(\ell_{24})+\cot(\ell_{34})}{1+(\cot(\ell_{24})+\cot(\ell_{34}))^2}=\frac{3}{\cot(\ell_{24})+\cot(\ell_{34})}.$$
However that equation clearly does not have any solution.

 We conclude that without taking into account delta couplings, the maximal multiplicity of the second eigenvalue for non-planar graphs may be less than $4$.

\section{Weighted delta couplings} \label{SecWDC}

Let us study whether other vertex conditions than standard delta couplings can be used to generalise
Colin de Verdi{\`e}re theory. 
 For a function $ u \in W_2^2 (\Gamma)  $ let $\vu^j$ 
 be the $d_j$-dimensional vector of values of $u$ at $v^j$ and $\de\vu^j$ to be the $d_j$-dimensional  vector of normal derivatives at $v^j$:
 \begin{equation}
 \vu^j = \{ u(x_i) \}_{x_i \in v^j}, \quad \quad \de\vu^j = \{ (-1)^{i+1} u' (x_i) \}_{x_i \in v^j}. 
 \end{equation}
Here $ d_j $ is the degree of the vertex $ v^j$ -- the number of edges connected at $ v^j$.

Then the most general vertex conditions are given by $ d_j \times d_j $ dimensional matrices $ A^j $ and $ B^j $ 
\begin{equation}
A^j \vu^j = B^j \de\vu^j, \quad j =1,2, \dots, M.
\end{equation}
The matrices $ A^j, B^j $ should satisfy the two conditions
$$  {\rm rank}\, (A^j, B^j) = d_j, \quad A^j (B^j)^* = B^j (A^j)^*. $$
Moreover it is important that the matrices properly reflect the connectivity of the graph, {\it i.e.} the same conditions 
at  a vertex $ v^j $ cannot be written with $ A^j $ and $ B^j $ having
the same block-diagonal form and therefore corresponding to a graph where the vertex $ v^j $ is chopped into two vertices.

The corresponding operator is self-adjoint, bounded from below and has discrete spectrum. Spectral properties of this operator are described in detail in \cite{BeKu,PK24}.

We have already mentioned that Colin de Verdi{\`e}re theory uses Perron-Frobenius theorem and   nodal domains of the eigenfunctions.
Perron-Frobenius property of Laplacians on metric graphs is governed by
 the vertex conditions, all vertex conditions leading to positive ground state have been characterised in \cite{Ku19,PK24}.
These conditions, called {\bf generalised delta couplings}, require in particular that all possible values of $ \vu^j$
are spanned by several vectors with positive coordinates and disjoint supports. In order to be able to introduce nodal domains we need that function values
at each vertex have a definite sign, therefore only weighted delta couplings with function values proportional to a single
vector with positive entries should be allowed. One may say that functions satisfy weighted continuity condition in this case, hence
the corresponding vertex conditions will be called  weighted delta couplings.

\begin{definition} \label{defgdc}
Let $ \vec{c}^j $ be a $d_j$-dimensional vector with positive entries and $ \alpha^j  \in \mathbb R $ be a real coupling parameter. Then  
{\bf weighted delta coupling} at the vertex $ v^j $ is given by the following conditions:
\begin{equation} \label{14}
\left\{
\begin{array}{l}
\vu^j \parallel \vec{c}^j,  \\[3mm]
\langle \vec{c}^j, \de\vu^j \rangle = \alpha^j  \frac{\langle \vec{c}^j, \vu^j \rangle}{\|\vec{c}^j\|^2}.
\end{array} \right.
\end{equation}
\end{definition}

The first condition can be seen as weighted continuity at the vertex.
It is natural to introduce the average value of the function at the vertex:
\begin{equation}
u_w (v^j) := \frac{ \langle \vec{c}^j, \vec{u}^j \rangle}{\| \vec{c}^j \|^2}. 
\end{equation}
Taking into account \eqref{14} we have
$ u_w(v^j)  = \frac{1}{\|\vec{c}^j\|^2}\sum_{x_i \in v^j} (\vec{c}^j)_i u(x_i) = \frac{u(x_{i_0})}{(\vec{c}^j)_{i_0}} ,$
for any $ i_0 \in v^j $.
In what follows we shall always assume that weighted delta conditions are satisfied at each vertex of $ \Gamma$. 
One obtains standard delta conditions by letting $ \vec{c}^j = (1,1, \dots, 1), \; \forall j. $

To introduce the M-function we define
\begin{equation}
\partial u_w (v^j) = \langle \vec{c}^j, \vec{u}^j \rangle, \quad m = 1,2, \dots, M.
\end{equation}
Then the M-function is given by:
\begin{equation}  \label{eqMw}
\small \mathbf M (\lambda):
\left(
\begin{array}{c}
u_w (v^1)  \\
u_w (v^2)  \\
\vdots \\
u_w (v^M) 
\end{array}
\right)  \mapsto \left(
\begin{array}{c}
\partial u_w (v^1) \\
\partial u_w (v^2)  \\
\vdots \\
\partial u_w (v^M) \end{array}
\right) ,
\end{equation}
instead of \eqref{eqM}.  It will be convenient to introduce the $ M \times M $ matrix $ \mathbf C $ determined by the vectors $  \vec{c}^i $
as follows
\begin{equation}
 c_{ij} =
 \left\{
 \begin{array}{ll}
 (\vec{c}^i)_n, & \mbox{if $ x_n \in v^i $ is the end point on the edge to  $ v^j$, } \\[3mm]
 0, & \mbox{if there is no edge between $ v^i $ and $ v^j$}.
 \end{array}
 \right.
\end{equation}
Repeating the calculations from Section \ref{Sec3} we get the M-function for the graph with weighted continuity conditions:
\begin{equation}
\small \left(\mathbf M (\lambda) \right)_{ij} =
\left\{
\begin{array}{ll}
- \displaystyle \sum_{v^m \sim v^i} c_{im}^2 \; \sqrt{\lambda} \cot \sqrt{\lambda} \ell_{im}, & j =i; \\[3mm]
 \displaystyle c_{ij} c_{ji} \; \frac{\sqrt{\lambda}}{\sin \sqrt{\lambda} \ell_{ij}} , & v^i \sim v^j; \\[3mm]
\displaystyle  0, &  v^i \not\sim v^j,j \neq i. \\[3mm]
 \end{array} \right.
\end{equation}

Lemma \ref{lem1} holds without any modification, while in Lemma \ref{lem2} one has to modify the asymptotics \eqref{asev} as
$
%\displaystyle
 \xi_j \sim - \left(\sum_{v^j \sim v^i} c_{ij}^2\right) \sqrt{|\lambda|}, \quad \lambda \rightarrow - \infty.
$

\section{Equilateral graphs with weighted continuity at the vertices} \label{SecEL}

In this section we shall discuss how to check planarity of equilateral metric graphs with weighted
delta couplings. 
Without loss of generality we assume that all edges have length equal to $1$:
$ \ell_{ij} = 1 .
$
An important class of weighted continuity conditions is when we associate one and the same weight with every
edge, {\it i.e.} the entries of the matrix $ C $ are chosen so that
$
c_{ij} = c_{ji}.
$
This will allow us to keep the number of parameters 
as in \eqref{parameters}.

In the considered case the singularities of the M-function always lie above the second eigenvalue $ \lambda_2 (\Gamma).$

\begin{lemma} \label{lemmaW}
Let $\Gamma$ be an equilateral graph (lengths $\ell_{ij}=1$) with weighted delta couplings. If $\Gamma$ is not given by a single interval then
$\lambda_2<\pi^2.$
\end{lemma}
\begin{proof}
Every graph in question has at least two edges, say $ E_1 = [x_1, x_2] $ and $ E_2= [x_3, x_4] $, since we do not allow loops.
Consider the test function $ \varphi $ not identically equal to zero just on these two edges:
$$ \small \varphi (x) = \left\{
\begin{array}{ll}
h_1 \sin \pi (x-x_1), &  x \in e_1, \\[3mm]
h_2 \sin \pi (x-x_3), &  x \in e_2, \\[3mm]
0, & \mbox{otherwise}.
\end{array} \right.
$$
The function is continuous on $ \Gamma$.
The constants $ h_1, h_2 $ can be chosen non-zero ensuring that the orthogonality conditions $ \int_\Gamma \varphi (x) \psi_1 (x) dx = 0 $
is satisfied. The Rayleigh quotient for $ \varphi $ is equal to $ \pi^2$ 
and therefore $\lambda_2\leq \pi^2$. To show that the inequality is strict it is enough to show that $ \varphi $ is not an eigenfunction. Since parallel edges are not allowed.  
at least one of the end points of $ E_1 $ is not equivalent to an end point in $ E_2$. At the vertex not containing an endpoint of $E_2$ the balance condition on the derivatives is not satisfied.
Hence $ \lambda_2 $ is strictly smaller than $ \pi^2$.
\end{proof}

One may get the same result by looking at the M-function.
Consider the limit 
$ \lim_{\lambda \rightarrow\pi^2}\frac{\pi-\sqrt{\lambda }}{\sqrt{\lambda }}\mathbf{Q}(\lambda ), $
which is equal to the matrix with the values $c_{ij}c_{ji}$ outside the diagonal and $\sum_{j\sim i} c_{ij}^2$ on the diagonal. 
The quadratic form of the matrix is given by
$ \sum_{v^i\sim v^j} \left( c_{ij} v_i + c_{ji} v_j \right)^2. $
Since $\Gamma$ is connected this matrix has at most one non-positive eigenvalue. 
Hence all eigenvalues  of $\mathbf{Q}(\lambda )$, except possibly one,
 tend to $\infty$ as $\lambda \rightarrow\pi^2$. 
 The number of energy curves tending to $ + \infty $ as $ \lambda \rightarrow \pi^2-0 $ is at least $ M-1$, where $ M = \# \mbox{vertices}. $
 Hence $ \lambda_{M-1} < \pi^2,$
since at least $ M-1$ energy curves had to cross the axis $ \xi = 0 $ to the left of $ \pi^2$. 

In view of Lemma \ref{lemmaW} to define admissible metric graphs we do not need to require that singularities of
the M-functions are to the right of $ \lambda_2 (\Gamma)$.

\begin{definition}
An equilateral metric graph $ \Gamma $ with weights $\vec{c}^j$ and delta couplings $ \alpha_j $ at the vertices is called
{\bf admissible} if and only if
\begin{itemize}
\item the matrix $ \mathbf Q(\lambda_2(\Gamma)) = \mathbf M (\lambda_2(\Gamma)) - {\rm diag}\, \{ \alpha_j \} $
possesses the Strong Arnold Property.
\end{itemize}
\end{definition}

\begin{definition}
The Colin de Verdi\` ere parameter $ \mu_w (\Gamma) $ for weighted equilateral metric graphs is equal to
the maximal multiplicity of $\lambda_2 (\Gamma) $ for Laplacians with weighted delta couplings at the vertices when one varies
over all admissible equilateral metric graphs associated with the same discrete graph.
\end{definition}
 
\begin{theorem} The Colin de Verdi{\`e}re graph parameters $\mu(G)$ and  $ \mu_w (\Gamma)$ are equal, provided $ \Gamma $ is the 
equilateral metric graph associated with the discrete graph $ G$.
\end{theorem}

\begin{proof}
The argument that $ \mu(G)\geq \mu_w (G) $ is almost the same as in the proof of Theorem \ref{ThST}: one needs to take into account that
 $\lambda_2<\pi^2 $ implies that $ \sin \sqrt{\lambda}_2  >0 $ implying that non-diagonal elements of $ -\mathbf Q(\lambda_2) $
 are negative.

 It remains to show that $ \mu(G)\leq \mu_w (G) $.
Let $ \mathbf{A} $ be one of the optimal Colin de Verdi{\`e}re matrices. We are going to construct a metric graph with weighted  delta couplings
so that $ \lambda_2 = (\frac{\pi}{2})^2. $ To this end we choose:
\begin{equation}
c_{ij} = \sqrt{\frac{\pi}{2} \vert a_{ij} \vert} \quad \mbox{and} \quad \alpha_j = a_{jj}.
\end{equation}
Note that we get $ c_{ij} = c_{ji} $ as the matrix $\mathbf{A}$ is real symmetric.
We get that the matrix $ - \mathbf Q((\pi/2)^2) $ coincides with the matrix $ \mathbf{A} $.
It follows that $ (\pi/2)^2 $ is an eigenvalue of multiplicity $ \mu (G)$.
It is not a ground state since  $ \mathbf Q((\pi/2)^2)$ has a positive eigenvalue. 
The lowest singularity of $  \mathbf Q(\lambda) $ is the point $ \lambda = \pi^2$, hence
only one energy curve crosses the line $ \xi = 0 $ to the left of $ (\pi/2)^2$.
It follows that
$ (\pi/2)^2 $ is the second eigenvalue
$ \lambda_2=  (\pi/2)^2. $
$\mathbf Q(\lambda_2) $ has the Strong Arnold Property since $\mathbf{A}$ has it.
The first assertion is proven.
\end{proof}

We decided to use $ \lambda_2 = (\pi/2)^2 $, also any other $ 0 < \lambda < \pi^2 $ would work.
It allowed us to simplify formulas using the fact that $ \cot \pi/2 = 0 $.

\section{Schr\"odinger operators} \label{SecSch}

So far we have only looked at Laplacians which act as $-\frac{d^2}{d x^2}$ on each edge. However it is possible to also look at Schr\"odinger operators which are given by the differential expression
$$L_q = - \frac{d^2}{dx^2}+q(x)$$ 
where $q$ is a real valued function. Here we restrict to the case where $q$ is constant on each edge. We denote by $q_{ij}$ the value 
of $q$ on the edge between $v^i$ and $v^j$.
It is enough to consider
equilateral graphs with standard delta couplings.

The edge M-function for $ L_q$ on the edge $v^i$ to $v^j$ is 
$$ \small \mathbf M_{\bf e} (\lambda) : = \begin{pmatrix} -\sqrt{\lambda-q_{ij} }\cot(\sqrt{\lambda-q_{ij} }) & \frac{\sqrt{\lambda-q_{ij} }}{\sin(\sqrt{\lambda-q_{ij} })} \\ \frac{\sqrt{\lambda-q_{ij} }}{\sin(\sqrt{\lambda-q_{ij} })} & -\sqrt{\lambda-q_{ij} }\cot(\sqrt{\lambda-q_{ij} })\end{pmatrix}.$$
Hence the M-function for $ L_q (\Gamma) $ will be
\begin{equation}
\small \left(\mathbf M (\lambda) \right)_{ij} =
\left\{
\begin{array}{ll}
- \displaystyle \sum_{v^j \sim v^m}  \; \sqrt{\lambda-q_{mj}} \cot \sqrt{\lambda-q_{mj}}, & i =j; \\[3mm]
 \displaystyle \frac{\sqrt{\lambda-q_{ij}}}{\sin \sqrt{\lambda-q_{ij}}} , & v^i \sim v^j; \\[3mm]
\displaystyle  0, &  v^i \not\sim v^j, \; i \neq j; \\[3mm]
 \end{array} \right.
\end{equation}
and the spectrum of  $L_q$ is given by $\mathbf Q(\lambda)=\mathbf M(\lambda) - {\rm diag}\, \{ \alpha_j \}$ for $\lambda\neq q_{ij}+\pi^2n^2$, $n$ positive integer.

This time we may choose $ \lambda = 0 $ and we assume that the Colin de Verdi\`ere matrix $ \mathbf A $ is scaled so that $ | a_{ij} | > 1,$ provided $ v^i \sim v^j. $
Then to get $ -\mathbf Q(0) = \mathbf A $ we need to satisfy the following equation  instead of \eqref{equal}
$$ \small - \frac{\sqrt{-q_{ij}}}{\sin \sqrt{-q_{ij}}}=a_{ij}. $$
This is always possible since the range of the function $ \frac{y}{\sin y} $ on the interval $ (0,\pi) $ covers the interval $ (1, \infty). $
Note that such $ q_{ij} $ always satisfy
$ - \pi^2 < q_{ij} < 0 . $
The singularities of the M-function lie at the points $ \lambda = q_{ij} + \pi^2 n^2, \; n =1,2, \dots ,$ which all are positive. Moreover $ \lambda = 0 $ is the second largest eigenvalue of $
- \mathbf A $, hence $ \lambda = 0 $ is the second (smallest) eigenvalue of $ L_q $.  We get an interpretation of the Colin de Verdi\`ere parameter as the largest multiplicity of the Schr\"odinger 
equation on a equilateral graph with delta couplings at the vertices.

The discussion above may be generalised by considering magnetic Schr\"odinger operators with delta couplings at the vertices given by
the differential expression
$$L_{q,a} =\left(i\frac{d}{dx}+a(x)\right)^2+q(x).$$ 
Variation of the real magnetic potential $ a$ gives another $ N $ parameters equal to the integrals of the magnetic potential along the edges.
This approach leads to the parameter $\nu$ introduced by Colin de Verdi\`ere  in \cite{YC98}.

It is also possible to put all three families together and consider Schr\"odinger operators with not edge-wise constant potentials on
non-equilateral graphs with weighted delta couplings at the vertices. Despite much larger number of parameters, this family does  lead to 
the same maximal multiplicity
 of the second eigenvalue.

 \section*{Acknowledgements}
 
 We would like to thank the anonymous referee not only  for pointing few corrections improving the presentation, but for suggesting
 to study further the Strong Arnold Property in the context of metric graphs. Our results on this subject are added to Section 4.

 \section*{Data availability statement}
 
 Data sets generated during the current study are available from the corresponding author on reasonable request.
 
 \section*{Ethics Declarations}
 
 Our institution does not require ethics approval for reporting individual cases or case
series. The Authors declare that there is no conflict of interest.
This work was partially supported by the Swedish Research Council Grants 2020-03780 and 2024-04650.

\end{document}